\documentclass[11pt,letterpaper]{amsart}

\usepackage{amssymb}
\usepackage{enumerate}
\usepackage{bbm}

\newtheorem{theorem}{Theorem}[section]
\newtheorem{proposition}[theorem]{Proposition}
\newtheorem{corollary}[theorem]{Corollary}
\newtheorem{lemma}[theorem]{Lemma}

\newcommand{\Z}{\mathbb{Z}}
\newcommand{\F}{\mathbb{F}}
\newcommand{\C}{\mathbb{C}}
\newcommand{\R}{\mathbb{R}}
\newcommand{\N}{\mathbb{N}}

\newcommand{\e}{\epsilon}

\newcommand{\1}{\mathbbm{1}}

\newcommand{\norm}[1]{\lVert#1\rVert}
\newcommand{\abs}[1]{\lvert#1\rvert}
\newcommand{\bigabs}[1]{\big\lvert#1\big\rvert}
\newcommand{\biggabs}[1]{\bigg\lvert#1\bigg\rvert}
\newcommand{\Biggabs}[1]{\Bigg\lvert#1\Bigg\rvert}

\newcommand{\sums}[1]{\sum_{\substack{#1}}}

\DeclareMathOperator{\Tr}{Tr}

\begin{document}

\title[Merit factors of polynomials derived from difference sets]{Merit factors of polynomials derived\\from difference sets}

\author{Christian G\"unther}
\address{Department of Mathematics, Paderborn University, Warburger Str.\ 100, 33098 Paderborn, Germany.}
\email[Ch. G\"unther]{chriguen@math.upb.de}

\author{Kai-Uwe Schmidt}
\address{Department of Mathematics, Paderborn University, Warburger Str.\ 100, 33098 Paderborn, Germany.}
\email[K.-U. Schmidt]{kus@math.upb.de}

\thanks{The authors are supported by German Research Foundation (DFG)}

\date{20 April 2015 (revised 11 February 2016)}


\begin{abstract}
The problem of constructing polynomials with all coefficients $1$ or $-1$ and large merit factor (equivalently with small $L^4$ norm on the unit circle) arises naturally in complex analysis, condensed matter physics, and digital communications engineering. Most known constructions arise (sometimes in a subtle way) from difference sets, in particular from Paley and Singer difference sets. We consider the asymptotic merit factor of polynomials constructed from other difference sets, providing the first essentially new examples since 1991. In particular we prove a general theorem on the asymptotic merit factor of polynomials arising from cyclotomy, which includes results on Hall and Paley difference sets as special cases. In addition, we establish the asymptotic merit factor of polynomials derived from Gordon-Mills-Welch difference sets and Sidelnikov almost difference sets, proving two recent conjectures.
\end{abstract}

\maketitle

\thispagestyle{empty}

\enlargethispage{1ex}

\vspace{-1.5ex}


\section{Introduction}

The problem of constructing polynomials having all coefficients in the set $\{-1,1\}$ (frequently called \emph{Littlewood polynomials}) with small $L^\alpha$ norm on the complex unit circle arises naturally in complex analysis~\cite{Lit1966},~\cite{Lit1968},~\cite{Bor2002},~\cite{Erd2002}, condensed matter physics~\cite{Ber1987}, and the design of sequences for communications devices~\cite{Gol1972},~\cite{BeeClaHer1985}.
\par
Recall that, for $1\le\alpha<\infty$, the $L^\alpha$ norm on the unit circle of a polynomial $f\in\C[z]$ is
\[
\norm{f}_\alpha=\left(\frac{1}{2\pi} \int_0^{2 \pi} \bigabs{f(e^{i\phi})}^\alpha \,d\phi\right)^{1/\alpha}.
\]
The $L^4$ norm has received particular attention because it is easier to calculate than most other $L^\alpha$ norms. Specifically, the $L^4$ norm of $f\in\C[z]$ is exactly the sum of the squared magnitudes of the coefficients of $f(z)\overline{f(z^{-1})}$. It is customary (see~\cite{Bor2002}, for example) to measure the smallness of the $L^4$ norm of a polynomial $f$ by its \emph{merit factor} $F(f)$, defined by
\[
F(f)=\frac{\norm{f}_2^4}{\norm{f}_4^4-\norm{f}_2^4},
\]
provided that the denominator is nonzero. Note that, if $f$ is a Littlewood polynomial of degree $n-1$, then $\norm{f}_2=\sqrt{n}$ and so a large merit factor means that the $L^4$ norm is small. 
\par
We note that Golay's original equivalent definition~\cite{Gol1972} of the merit factor involves the aperiodic autocorrelations (which are precisely the coefficients of $f(z)\overline{f(z^{-1})}$) of the sequence formed by the coefficients of~$f$.
\par
Besides continuous progress on the merit factor problem in the last fifty years (see~\cite{Jed2005},~\cite{Hoh2006},~\cite{BorFerKna2008} for surveys and~\cite{JedKatSch2013} for a brief review of more recent work), modulo generalisations and variations, only three nontrivial families of Littlewood polynomials are known, for which we can compute the asymptotic merit factor. These are: Rudin-Shapiro polynomials and polynomials whose coefficients are derived either from multiplicative or additive characters of finite fields. As shown by Littlewood~\cite{Lit1968}, Rudin-Shapiro polynomials are constructed recursively in such a way that their merit factor satisfies a simple recurrence, which gives an asymptotic value of $3$. The largest asymptotic merit factors that have been obtained from the other two families equal cubic algebraic numbers $6.342061\dots$ and $3.342065\dots$, respectively~\cite{JedKatSch2013a},~\cite{JedKatSch2013} (see Corollary~\ref{cor:second_order} and Theorem~\ref{thm:gmw} in this paper 
for precise statements).
\par
The latter polynomials are closely related to classical difference sets, namely Paley and Singer difference sets. Recall that a \emph{difference set} with parameters $(n,k,\lambda)$ is a $k$-subset $D$ of a finite group $G$ of order $n$ such that the $k(k-1)$ nonzero differences of elements in $D$ hit every nonzero element of~$G$ exactly $\lambda$ times (so that $k(k-1)=\lambda(n-1)$). We are interested in the case that the group $G$ is cyclic. In this case, we fix a generator $\theta$ of $G$ and associate with a subset $D$ of $G$ the Littlewood polynomial
\begin{equation}
f_{r,t}(z)=\sum_{j=0}^{t-1}\1_D(\theta^{j+r})z^j,   \label{eqn:set_to_poly}
\end{equation}
where $r$ and $t$ are integers with $t\ge 0$ and
\[
\1_D(y)=\begin{cases}
 1 & \text{for $y\in D$}\\
-1 & \text{for $y\in G\setminus D$}.
\end{cases}
\]
If $n$ is the group order, then $f_{0,n}$ captures the information about~$D$. We call this polynomial a \emph{characteristic polynomial} of~$D$ (which is unique up to the choice of $\theta$) and the polynomials $f_{r,n}$ \emph{shifted} characteristic polynomials.
\par
The results of this paper are mainly motivated by the 1991 paper of Jensen, Jensen, and H{\o}holdt~\cite{JenJenHoh1991}, in which the authors asked for the merit factor of polynomials derived from families of difference sets. It was shown in~\cite{JenJenHoh1991} that, among the known families of difference sets, only those with Hadamard parameters, namely
\[
(4h-1,\,2h-1,\,h-1)
\]
for a positive integer $h$, can give a nonzero asymptotic merit factor for their shifted characteristic polynomials. As of 1991, five families of such difference sets were known~\cite{JenJenHoh1991}:
\begin{enumerate}[(A)]
\item Paley difference sets,
\item Singer difference sets,
\item Twin-prime difference sets,
\item Gordon-Mills-Welch difference sets,
\item Hall difference sets.
\end{enumerate}
While in the first three cases, the asymptotic merit factors of the shifted characteristic polynomials have been determined in~\cite{HohJen1988} and~\cite{JenJenHoh1991} and those of the more general polynomials~\eqref{eqn:set_to_poly} in~\cite{JedKatSch2013a} and~\cite{JedKatSch2013}, the last two cases were left as open problems in~\cite{JenJenHoh1991}. More specifically, the asymptotic merit factor of the polynomials derived from a subclass of the Gordon-Mills-Welch difference sets is subject to a conjecture~\cite[Conjecture~7.1]{JedKatSch2013}. In this paper, we prove this conjecture and solve the problems concerning Gordon-Mills-Welch and Hall difference sets posed in~\cite{JenJenHoh1991} (see Theorem~\ref{thm:gmw} and Corollary~\ref{cor:sixth_order}, respectively). In addition, we obtain the asymptotic merit factor of polynomials related to a construction of Sidelnikov~\cite{Sid1969} (see also~\cite{LemCohEas1977}). This explains numerical observations in~\cite{HarYaz2010} and proves in the 
affirmative~\cite[Conjecture~7.2]{JedKatSch2013}.
\par
In fact, the result for Hall difference sets arises from a much more general theorem concerning polynomials derived from cyclotomy (see Theorem~\ref{thm:mf_cyclotomic}). This result considers polynomials constructed from subsets of~$\F_p$ obtained by joining $m/2$ of the $m$ cyclotomic classes of (even) order $m$, where $m$ satisfies $p\equiv 1\pmod m$. The cases $m\in\{2,4,6\}$ are examined in detail. For $m=2$, we obtain the asymptotic merit factor of polynomials arising from Paley difference sets (see Corollary~\ref{cor:second_order}), which is the main result of~\cite{JedKatSch2013a}. For $m=4$, we obtain the asymptotic merit factor of polynomials arising from Ding-Helleseth-Lam almost difference sets~\cite{DinHelLam1999} (see Corollary~\ref{cor:fourth_order}). For $m=6$, we obtain, among other things, the asymptotic merit factor of polynomials arising from Hall difference sets (see Corollary~\ref{cor:sixth_order}).
\par
Some comments on our result for Gordon-Mills-Welch difference sets follow. In the cyclic case, such sets have paramaters
\[
(2^m-1,2^{m-1}-1,2^{m-2}-1),
\]
which are typically called \emph{Singer} parameters. The Gordon-Mills-Welch construction produces difference sets in a cyclic group $G$ of order $2^m-1$ from difference sets with Singer parameters in a subgroup of $G$. Hence this construction is very general and can in particular be iterated. Our result on Gordon-Mills-Welch difference sets (Theorem~\ref{thm:gmw}) requires no knowledge about the smaller difference sets that are used as building blocks. Thus Singer, Paley, or Hall difference sets (in groups whose order is a Mersenne number) can be used as building blocks. In addition, since 1991, further families of difference sets in cyclic groups with Singer parameters have been found:
\begin{enumerate}[(A)]
\setcounter{enumi}{5}
\item Maschietti difference sets~\cite{Mas1998},
\item Dillon-Dobbertin difference sets~\cite{DilDob2004},
\item No-Chung-Yun difference sets~\cite{DilDob2004}.
\end{enumerate}
Our results include the cases when these difference sets are used as building blocks in the Gordon-Mills-Welch construction. However we have not been able to determine the asymptotic merit factors of the polynomials associated with these difference sets themselves. We conjecture that they have the same behaviour as those of Singer and Gordon-Mills-Welch difference sets, given in Theorem~\ref{thm:gmw}.


\section{Results}
\label{sec:results}

To state our results, we require the function $\varphi_\nu:\R\times\R^{+}\to\R$, defined for real $\nu$ by
\begin{align*}
\frac{1}{\varphi_\nu(R,T)}=
1-\frac{2(1+\nu)T}{3}+4&\sum_{m\in\N}\max\bigg(0,1-\frac{m}{T}\bigg)^2\\
+\nu &\sum_{m\in\Z}\max\bigg(0,1-\biggabs{1+\frac{2R-m}{T}}\bigg)^2,
\end{align*}
where $\N$ is the set of positive integers. This function satisfies $\varphi_\nu(R,T)=\varphi_\nu(R+\tfrac{1}{2},T)$ on its entire domain. It will be useful to know the global maximum of $\varphi_\nu$ for certain values of $\nu$. The function $\varphi_1$ was maximised in~\cite[Corollary 3.2]{JedKatSch2013}. Using the same approach, we find that, for all $\nu\in[0,1]$, the global maximum of $\varphi_\nu(R,T)$ exists and equals the largest root of
\begin{multline*}
(\nu^4-2\nu^3-3\nu^2-50\nu+112)X^3+(12\nu^3+36\nu^2-18\nu-528)X^2\\ +(24\nu^2+282\nu+528)X-6\nu-48.
\end{multline*}
The global maximum is unique for $R\in[0,\tfrac{1}{2})$ and is attained when $T$ is the middle root of
\[
(2\nu+2)X^3-(6\nu+24)X+3\nu+24
\]
and $R=3/4-T/2$.
\par
We begin with stating our results for Gordon-Mills-Welch difference sets \cite{GorMilWel1962} whose ambient group is $\F_q^*$, where $q>2$ is a power of two\footnote{We note that~\cite{GorMilWel1962} defines more general difference sets, which are also called Gordon-Mills-Welch difference sets. However, the sets considered in this paper are the only ones with Hadamard parameters.}. Let~$\F_s$ be a proper subfield of $\F_q$ and let $A$ contain all elements $a\in\F_q$ with $\Tr_{q,s}(a)=1$, where $\Tr_{u,v}$ is the trace from $\F_u$ to $\F_v$. Let $B$ be a difference set in $\F_s^*$ with $\abs{B}=s/2$ (so that, for $s>2$, the complement of $B$ in $\F_s^*$ has Singer parameters). We also allow $s=2$, so that $B$ is a trivial difference set. A set of the form
\begin{equation}
\{ab:a\in A,b\in B\}   \label{eqn:gmw_construction}
\end{equation}
is a \emph{Gordon-Mills-Welch difference set} in $\F_q^*$ (whose complement has Singer parameters). They generalise the Singer difference sets, which arise for $s=2$. We have the following result for the asymptotic merit factor of polynomials obtained by Gordon-Mills-Welch difference sets.
\begin{theorem}
\label{thm:gmw}
Let $q>2$ be a power of two and let $f$ be a characteristic polynomial of a Gordon-Mills-Welch difference set in $\F_q^*$. Let $T>0$ be real. If $t/q\to T$, then $F(f_{r,t})\to \varphi_0(0,T)$ as $q\to\infty$.
\end{theorem}
\par
In the particular case of Singer difference sets, Theorem~\ref{thm:gmw} reduces to~\cite[Theorem~2.2~(i)]{JedKatSch2013}. The case that the Gordon-Mills-Welch difference sets in Theorem~\ref{thm:gmw} are of the form~\eqref{eqn:gmw_construction} when $B$ is a Singer difference set proves~\cite[Conjecture~7.1]{JedKatSch2013}\footnote{Conjecture~7.1 of \cite{JedKatSch2013} also involves ``negaperiodic'' and ``periodic'' extensions of the polynomials associated with Gordon-Mills-Welch difference sets. The corresponding assertions can be obtained as direct consequences of Proposition~\ref{pro:L_gmw} and~\cite[Theorem~4.2]{JedKatSch2013}, but are omitted here for the sake of simplicity.}. 
\par
Next we consider subsets of $\F_q^*$ for an odd prime power $q$, which are related to a construction of Sidelnikov~\cite{Sid1969}. We call a set of the form
\begin{equation}
\{x\in\F_q^*:\text{$x+1$ is zero or a square in $\F_q^*$}\}   \label{eqn:def_sidelnikov}
\end{equation}
a \emph{Sidelnikov set} in $\F_q^*$. Such a set gives rise to a so-called almost difference set~\cite[Theorem~4]{AraDinHelKumMar2001}. We have the following result for the asymptotic merit factor of the associated polynomials, proving~\cite[Conjecture~7.2]{JedKatSch2013}.
\begin{theorem}
\label{thm:sidelnikov}
Let $q$ be an odd prime power and let $f$ be a characteristic polynomial of a Sidelnikov set in $\F_q^*$. Let $T>0$ be real. If $t/q\to T$, then $F(f_{r,t})\to \varphi_0(0,T)$ as $q\to\infty$.
\end{theorem}
\par
The maximum asymptotic merit factor that can be obtained in Theorems~\ref{thm:gmw} and~\ref{thm:sidelnikov} is $3.342065\dots$, the largest root of
\[
7X^3-33X^2+33X-3.
\]
\par
Next we construct Littlewood polynomials using cyclotomy. Let $m$ be a positive integer and let $p$ be a prime satisfying $p\equiv 1\pmod m$. Let $\omega$ be a fixed primitive element in $\F_p$. Let $C_0$ be the set of $m$-th powers in $\F_p^*$ and write $C_s=\omega^sC_0$ for $s\in\Z$. The sets $C_0,C_1,\dots,C_{m-1}$ partition $\F_p^*$ and are called the \emph{cyclotomic classes} of $\F_p$ of order $m$.
\par
We construct subsets $D$ of the additive group $\F_p$ by joining some of these classes. This method provides a rich source of difference sets (see~\cite{Jun1992} for a survey). We may take $1$ as a generator for $\F_p$, in which case the polynomials associated with $D$ are
\[
f_{r,t}(z)=\sum_{j=0}^{t-1}\1_D(j+r)z^j
\]
and a characteristic polynomial is $f_{0,p}$ (this is no loss of generality; if the generator is $v$, then replace $D$ by $v^{-1}D$). It follows from~\cite[Theorem~2.1]{JenJenHoh1991} that the shifted characteristic polynomials associated with $D$ have a nonzero asymptotic merit factor only if $\abs{D}/p$ approaches $1/2$ as $p\to\infty$, thus $m$ must be even and $D$ must be a union of $m/2$ cyclotomic classes. Two families of difference sets arise in this way, namely the Paley difference sets for $m=2$ and the Hall difference sets for $m=6$~\cite[Theorem~2.2]{Hal1956}. If $D$ is a union of $m/2$ cyclotomic classes of order $m$, then $\abs{D}=(p-1)/2$, so if $D$ is a difference set, then it must have Hadamard parameters. Equivalently, $\abs{(D+u)\cap D}=(p-3)/4$ for every $u\in\F_p^*$. Our next theorem applies not only to such difference sets, but requires this condition to hold asymptotically (in a precise sense).
\begin{theorem}
\label{thm:mf_cyclotomic}
Let $m$ be an even positive integer and let $S$ be an $m/2$-element subset of $\{0,1,\dots,m-1\}$. Let $p$ take values in an infinite set of primes satisfying $p\equiv 1\pmod m$. Let $D$ be the union of the $m/2$ cyclotomic classes $C_s$ with $s\in S$ of $\F_p$ of order $m$ and suppose that, as $p\to\infty$,
\begin{equation}
\frac{(\log p)^3}{p^2}\sum_{u\in\F_p^*}\bigg(\bigabs{(D+u)\cap D}-\frac{p}{4}\bigg)^2\to 0.   \label{eqn:cond_Ru}
\end{equation}
Let $f$ be a characteristic polynomial of $D$ and let $R$ and $T>0$ be real. If $r/p\to R$ and $t/p\to T$, then the following hold as $p\to\infty$:
\begin{enumerate}[(i)]
\setlength{\itemsep}{1ex}
\item If $\frac{p-1}{m}$ is even for every $p$, then $F(f_{r,t})\to\varphi_1(R,T)$.
\item If $\frac{p-1}{m}$ is odd for every $p$, then $F(f_{r,t})\to\varphi_\nu(R,T)$, where $\nu=(\tfrac{4N}{m}-1)^2$ and
\[
N=\bigabs{\{(s,s')\in S\times S:s-s'=m/2\}}.
\]
\end{enumerate}
\end{theorem}
\par
Several remarks on Theorem~\ref{thm:mf_cyclotomic} follow. It is readily verified that $\nu$ in Theorem~\ref{thm:mf_cyclotomic} satisfies $\nu\in[0,1]$. The condition~\eqref{eqn:cond_Ru} is essentially necessary since
\[
\frac{1}{F(f)}\ge\frac{8}{p^2}\sum_{u\in \F_p^*}\bigg(\bigabs{(D+u)\cap D}-\frac{p-2}{4}\bigg)^2.
\]
This can be deduced from the proof of Theorem~\ref{thm:mf_cyclotomic} and the inequality
\[
\norm{f}_4^4\ge\frac{1}{2p}\sum_{k\in\F_p}\abs{f(e^{2\pi ik/p})}^4+\frac{p^2}{2},
\]
which can be obtained from~\cite[(2.3)]{HohJen1988} with an extra step involving the Cauchy-Schwarz inequality. In fact, this is a refinement of the Marcinkiewicz-Zygmund inequality~\cite[Chapter~X, Theorem~7.5]{Zyg1959} for the $L^4$ norm.
\par
The condition~\eqref{eqn:cond_Ru} can be checked using the cyclotomic numbers of order~$m$, which are the $m^2$ numbers
\[
\abs{(C_i+1)\cap C_j}
\]
for $0\le i,j<m$. These numbers are known explicitly for all even $m\le 20$ and for $m=24$ (see~\cite[p.~152]{BerEvaWil1998} for a list of references). Also note that the conclusion of Theorem~\ref{thm:mf_cyclotomic} remains unchanged if we replace $S$ by $h+S$ reduced modulo~$m$ for an integer $h$ (which changes $D$ to $\omega^hD$).
\par
We now consider in detail the cases $m\in\{2,4,6\}$ of Theorem~\ref{thm:mf_cyclotomic}. If $m=2$, then $D$ consists of either the squares or the nonsquares of $\F_p^*$. In both cases, $D$ is a Paley difference set for $p\equiv 3\pmod 4$. As remarked above, we can assume without loss of generality that $D$ is the set of squares in $\F_p^*$. Then we have (see~\cite[Theorem~2.2.2]{BerEvaWil1998}, for example)
\[
4\,\abs{(D+u)\cap D}=\begin{cases}
p-4-(-1)^{\frac{p-1}{2}} & \text{for $u$ a square in $\F_p^*$}\\[.5ex]
p-2+(-1)^{\frac{p-1}{2}} & \text{for $u$ a nonsquare in $\F_p^*$}.
\end{cases}
\]
Noting that $\nu=1$ for $m=2$, we obtain the following corollary, which is essentially the main result of~\cite{JedKatSch2013a} (see also~\cite[Theorem~2.1]{JedKatSch2013}).
\begin{corollary}
\label{cor:second_order}
Let $p$ take values in an infinite set of odd primes, let $D$ be either the set of squares or the set of nonsquares of $\F_p^*$ and let $f$ be a characteristic polynomial of $D$. Let $R$ and $T>0$ be real. If $r/p\to R$ and $t/p \to T$, then $F(f_{r,t})\to \varphi_1(R,T)$ as $p\to\infty$.
\end{corollary}
\par
We now look at the case $m=4$. Here, we only need to consider two cases for joining two cyclotomic classes of order four, namely $C_0\cup C_2$ and $C_0\cup C_1$. The first case brings us back to $m=2$. When $p$ is of the form $x^2+4$ for $x\in\Z$ and $(p-1)/4$ is odd, the second case gives rise to the Ding-Helleseth-Lam almost difference sets~\cite[Theorem~4]{DinHelLam1999}. By inspecting the cyclotomic numbers of order four (see Section~\ref{sec:residues}), we obtain the following.
\begin{corollary}
\label{cor:fourth_order}
Let $p$ take values in an infinite set of primes of the form $x^2+4y^2$ for $x,y\in\Z$ such that $y^2(\log p)^3/p\to 0$ as $p\to\infty$. Let $D$ be the union of two cyclotomic classes of $\F_p$ of order four and let $f$ be a characteristic polynomial of $D$. Let $R$ and $T>0$ be real. If $r/p \to R$ and $t/p \to T$, then $F(f_{r,t})\to \varphi_1(R,T)$ as $p\to\infty$.
\end{corollary}
\par
Recall from elementary number theory that primes of the form $x^2+4y^2$ for $x,y\in\Z$ are exactly the primes that are congruent to $1$ modulo $4$. It is also known~\cite{Col1993} that there are infinitely many primes satisfying the hypothesis of the corollary. 
\par
The case $m=6$ is the first situation, where different limiting functions occur. In this case, there are four different sets $D$ to consider, namely
\begin{equation}
C_0\cup C_2\cup C_4,\quad C_0\cup C_1\cup C_2,\quad C_0\cup C_1\cup C_3,\quad C_0\cup C_1\cup C_4.   \label{eqn:sets_of_order_six}
\end{equation}
Again, the first set brings us back to $m=2$. When $p$ is of the form $x^2+27$ for $x\in\Z$ and $(p-1)/6$ is odd, then either the third or the fourth set in~\eqref{eqn:sets_of_order_six} gives rise to Hall difference sets~\cite{Hal1956} (the choice depends on the primitive element~$\omega$). We shall see that Theorem~\ref{thm:mf_cyclotomic} gives two possible limiting functions for the sixth cyclotomic classes, which is our motivation for the following definition. Let $D$ be a union of three cyclotomic classes of order six. If there is a $\gamma\in\F_p^*$ such that $\gamma D$ equals one of the first two sets in~\eqref{eqn:sets_of_order_six}, then we say that $D$ is of \emph{Paley type}. Otherwise, we say that $D$ is of \emph{Hall type}. 
\begin{corollary}
\label{cor:sixth_order}
Let $p$ take values in an infinite set of primes of the form $x^2+27y^2$ for $x,y\in\Z$ such that $y^2(\log p)^3/p\to 0$ as $p\to\infty$. Let $D$ be the union of three cyclotomic classes of $\F_p$ of order six and let $f$ be a characteristic polynomial of $D$. Let $R$ and $T>0$ be real. If $r/p\to R$ and $t/p\to T$, then the following hold as $p\to\infty$:
\begin{enumerate}[(i)]
\setlength{\itemsep}{1ex}
\item If, for each $p$, $D$ is of Paley type or $\frac{p-1}{6}$ is even, then $F(f_{r,t})\to \varphi_1(R,T)$.
\item If, for each $p$, $D$ is of Hall type and $\frac{p-1}{6}$ is odd, then $F(f_{r,t})\to \varphi_{1/9}(R,T)$.
\end{enumerate}
\end{corollary}
\par
It is known that primes of the form $x^2+27y^2$ for $x,y\in\Z$ are exactly the primes $p$ for which $p\equiv 1\pmod 6$ and $2$ is a cube modulo $p$~\cite[Corollary~2.6.4]{BerEvaWil1998}. Again, it is also known~\cite{Col1993} that there are infinitely many primes satisfying the hypothesis of the corollary.
\par
The largest asymptotic merit factor that can be obtained in Corollaries~\ref{cor:second_order},~\ref{cor:fourth_order}, and~\ref{cor:sixth_order}~(i) is $6.342061\dots$, the largest root of
\[
29X^3-249X^2+417X-27,
\]
which equals the best known asymptotic value for Littlewood polynomials. The largest asymptotic merit factor that can be obtained in Corollary~\ref{cor:sixth_order}~(ii) is $3.518994\dots$, the largest root of
\[
349061X^3-1737153X^2+1835865X-159651.
\]
\par
It is also of interest to look at the case $T=1$ in our results, which concerns just the shifted characteristic polynomials, as considered in~\cite{HohJen1988} and~\cite{JenJenHoh1991} for Paley and Singer difference sets, respectively. Since
\[
\frac{1}{\varphi_\nu(R,1)}=\tfrac{1}{6}(2-\nu)+8\nu(R-\tfrac{1}{4})^2\quad\text{for $0\le R\le \tfrac{1}{2}$},
\]
the global maximum of $g_\nu(R,1)$ equals $6/(2-\nu)$. Hence, for $T=1$, Theorems~\ref{thm:gmw} and~\ref{thm:sidelnikov} give an asymptotic merit factor of $3$, Corollaries~\ref{cor:second_order},~\ref{cor:fourth_order}, and~\ref{cor:sixth_order}~(i) give a maximum asymptotic merit factor of $6$ and Corollary~\ref{cor:sixth_order}~(ii) gives a maximum asymptotic merit factor of $54/17$.
\par
We note that there are also ``negaperiodic'' and ``periodic'' versions of Theorem~\ref{thm:gmw}, Theorem~\ref{thm:mf_cyclotomic}, and its corollaries, as considered in~\cite{JedKatSch2013} (but not of Theorem~\ref{thm:sidelnikov} since in this case the characteristic polynomials have odd degree). These follow directly from our results and a generalisation of Theorem~\ref{thm:mf_from_L_1} in the vein of parts (ii) and (iii) of Theorems~4.1 and 4.2 in~\cite{JedKatSch2013}. We omit their statements for the sake of simplicity.
\par
We shall prove Theorems~\ref{thm:gmw} and~\ref{thm:sidelnikov} in Sections~\ref{sec:gmw} and~\ref{sec:sidelnikov}, respectively. Theorem~\ref{thm:mf_cyclotomic} and Corollaries~\ref{cor:fourth_order} and~\ref{cor:sixth_order} will be proved in Section~\ref{sec:residues}.


\section{Asymptotic merit factor calculation}

Let $f(z)=\sum_{j=0}^{n-1}a_jz^j$ be a Littlewood polynomial of degree $n-1$ and let~$r$ and $t$ be integers with $t\ge 0$. Define the polynomial
\[
f_{r,t}(z)=\sum_{j=0}^{t-1}a_{j+r}z^j,
\]
where we extend the definition of $a_j$ so that $a_{j+n}=a_j$ for all $j\in\Z$. Write $\e_k=e^{2\pi ik/n}$. From~\cite{JedKatSch2013} it is known that $F(f_{r,t})$ depends only on the function $L_f:(\Z/n\Z)^3\to\Z$, defined by
\[
L_f(a,b,c)=\frac{1}{n^3}\sum_{k\in\Z/n\Z}f(\e_k)f(\e_{k+a})\overline{f(\e_{k+b})}\overline{f(\e_{k+c})}.
\]
Define the functions $I_n,J_n:(\Z/n\Z)^3\to\Z$ by
\begin{align*}
I_n(a,b,c)&=\begin{cases}
1 & \text{if ($c=a$ and $b=0$) or ($b=a$ and $c=0$),}\\
0 & \text{otherwise}
\end{cases}
\intertext{and}
J_n(a,b,c)&=\begin{cases}
1 & \text{if $a=0$ and $b=c\ne 0$,}\\
0 & \text{otherwise}
\end{cases}
\intertext{and, for even $n$, the function $K_n:(\Z/n\Z)^3\to\Z$ by}
K_n(a,b,c)&=\begin{cases}
1 & \text{if $a=n/2$ and $b=c+n/2$ and $bc\ne 0$,}\\
0 & \text{otherwise}.
\end{cases}
\end{align*}
In order to prove Theorems~\ref{thm:gmw},~\ref{thm:sidelnikov}, and~\ref{thm:mf_cyclotomic} we shall show that the corresponding function $L_f$ is well approximated by either $I_n+\nu J_n$ for an appropriate real $\nu$ or by $I_n+K_n$ and then apply one of the following two theorems. Our first theorem is a slight generalisation of Theorems~4.1~(i) and~4.2~(i) of~\cite{JedKatSch2013}, which arise by setting $\nu=1$ and $\nu=0$, respectively. This theorem can be proved by applying straightforward modifications to the proof of~\cite[Theorem~4.1]{JedKatSch2013}.
\begin{theorem}
\label{thm:mf_from_L_1}
Let $\nu$ be a real number and let $n$ take values in an infinite set of positive integers. For each $n$, let $f$ be a Littlewood polynomial of degree $n-1$ and suppose that, as $n\to\infty$,
\[
(\log n)^3\max\limits_{a,b,c\in\Z/n\Z}\bigabs{L_f(a,b,c)-(I_n(a,b,c)+\nu J_n(a,b,c))}\to 0.
\]
Let $R$ and $T>0$ be real. If $r/n\to R$ and $t/n\to T$, then $F(f_{r,t})\to\varphi_{\nu}(R,T)$ as $n\to\infty$.
\end{theorem}
\par
There is a similar generalisation of parts (ii) and (iii) of Theorems~4.1 and~4.2 in~\cite{JedKatSch2013}, which we do not consider in this paper.
\par
Our second theorem is a more subtle modification of Theorem~4.1~(i) in~\cite{JedKatSch2013}. We include a proof that highlights the required modifications of the proof of \cite[Theorem~4.1~(i)]{JedKatSch2013}.
\begin{theorem}
\label{thm:mf_from_L_2}
Let $n$ take values in an infinite set of even positive integers. For each $n$, let $f$ be a Littlewood polynomial of degree $n-1$ and suppose that, as $n\to\infty$,
\begin{equation}
(\log n)^3\max\limits_{a,b,c\in\Z/n\Z}\bigabs{L_f(a,b,c)-(I_n(a,b,c)+ K_n(a,b,c))}\to 0.   \label{eqn:assumption_LIK}
\end{equation}
Let $T>0$ be real. If $t/n\to T$, then $F(f_{r,t})\to \varphi_0(0,T)$ as $n\to\infty$.
\end{theorem}
\begin{proof}
The first part of the proof is identical to that of~\cite[Theorem~4.1~(i)]{JedKatSch2013}, giving
\begin{equation}
\frac{1}{F(f_{r,t})}=-1+\frac{1}{t^2n}\;\sums{0\le j_1,j_2,j_3,j_4<t \\ j_1+j_2=j_3+j_4}\;\sum_{a,b,c\in\Z/n\Z}\;L_f(a,b,c)\,\e_a^{-j_2-r}\e_b^{j_3+r}\e_c^{j_4+r},   \label{eqn:F_from_L}
\end{equation}
where $\e_k=e^{2\pi ik/n}$. Write
\begin{equation}
L_f(a,b,c)=I_n(a,b,c)+K_n(a,b,c)+M_n(a,b,c),   \label{eqn:IplusK}
\end{equation}
where $M_n(a,b,c)$ is an error term, which can be controlled using~\eqref{eqn:assumption_LIK}. Consider three cases for the tuple $(a,b,c)\in\Z/n\Z$: (1) $c=a$ and $b=0$, (2) $a=b$ and $c=0$, and (3) $b=c+n/2$ and $a=n/2$. Then $I_n(a,b,c)+K_n(a,b,c)$ equals $1$ if at least one of these conditions is satisfied and $I_n(a,b,c)+K_n(a,b,c)$ equals $0$ otherwise. There are exactly three tuples $(a,b,c)$ that satisfy more than one of these conditions, namely $(0,0,0)$, $(n/2,n/2,0)$, and $(n/2,0,n/2)$. 
\par
We now substitute~\eqref{eqn:IplusK} into~\eqref{eqn:F_from_L} and break the sum involving $I_n(a,b,c)+K_n(a,b,c)$ into six parts: three sums corresponding to the three cases and three sums to correct for the double counting of $(0,0,0)$, $(n/2,n/2,0)$, and $(n/2,0,n/2)$. Noting that the sums arising in cases (1) and (2) have the same value, as have the sums arising for the compensation of the double count of $(n/2,n/2,0)$ and $(n/2,0,n/2)$, we obtain
\[
\frac{1}{F(f_{r,t})}=-1+A+B+C-D_1-D_2-D_3+E,
\]
where
\begin{align*}
A=B&=\frac{1}{t^2n}\sums{0 \le j_1,j_2,j_3,j_4 < t \\ j_1+j_2=j_3+j_4}\;\sum_{b\in\Z/n\Z} \e_b^{j_3-j_2}, \\
C&=\frac{1}{t^2n}\sums{0 \le j_1,j_2,j_3,j_4 < t \\ j_1+j_2=j_3+j_4} (-1)^{j_3-j_2}\sum_{c\in\Z/n\Z} \e_c^{j_3+j_4+2 r}, \\
D_1&=\frac{1}{t^2n}\sums{0 \le j_1,j_2,j_3,j_4 < t \\ j_1+j_2=j_3+j_4} 1,\\
D_2=D_3&=\frac{1}{t^2n}\sums{0 \le j_1,j_2,j_3,j_4 < t \\ j_1+j_2=j_3+j_4} (-1)^{j_3-j_2},\\
E&=\frac{1}{t^2n}\sum_{a,b,c\,\in\,\Z/n\Z}M_n(a,b,c)\,\e_{-a+b+c}^r\,\sums{0 \le j_1,j_2,j_3,j_4 < t \\ j_1+j_2=j_3+j_4}\e_a^{-j_2}\e_b^{j_3}\e_c^{j_4}.
\end{align*}
As in the proof of~\cite[Theorem~4.2]{JedKatSch2013}, we have $-1+A+B-D_1+E\to 1/\varphi_0(0,T)$ if $t/n\to T$. Hence it remains to show that $C-D_2-D_3\to 0$ if $t/n\to T$. 
\par
Since there are contributions to the first sum in $C$ only when $j_3+j_4=mn-2r$ for some $m\in\Z$, we obtain
\[
C=\frac{1}{t^2}\sum_{m\in\Z}\Bigg(\sum_{0\le j,mn-2r-j<t}(-1)^j\Bigg)^2.
\]
Therefore we have $\abs{C}\le 1/(tn)$ and so $C\to 0$ if $t/n\to T$. By writing $j_3=j_1+m$ for some $m\in\Z$ we find that
\[
D_2=\frac{1}{t^2n}\sum_{m\in\Z}\Bigg(\sum_{0\le j,j+m<t}(-1)^j\Bigg)^2.
\]
Hence $\abs{D_2}\le 1/(tn)$ and therefore $D_2+D_3\to0$ if $t/n\to T$. This completes the proof.
\end{proof}


\section{Some background on Gauss and Jacobi sums}

Let $\chi$ be a multiplicative character of $\F_q$. Throughout this paper, we use the common convention
\[
\chi(0)=\begin{cases}
0 & \text{if $\chi$ is nontrivial}\\
1 & \text{if $\chi$ is trivial}.
\end{cases}
\]
We define the \emph{canonical Gauss sum} of $\chi$ to be
\begin{equation}
G(\chi)=\sum_{y\in\F_q^*}\chi(y)\,e^{2\pi i\Tr_{q,p}(y)/p},   \label{eqn:def_gauss}
\end{equation}
where $p$ is the characteristic of $\F_q$ and $\Tr_{q,p}$ is the trace from $\F_q$ to $\F_p$. Below we summarise some basic facts about such Gauss sums (see~\cite[Chapter~5]{LidNie1997} or~\cite[Chapter~1]{BerEvaWil1998}, for example).
\begin{lemma}
Let $\chi$ be a multiplicative character of $\F_q$. Then the following hold.
\begin{enumerate}[(i)]
\label{lem:gauss}
\setlength{\itemsep}{1ex}
\item $G(\chi)=-1$ if $\chi$ is trivial.
\item $\abs{G(\chi)}=\sqrt{q}$ if $\chi$ is nontrivial.
\item $G(\chi)G(\overline{\chi})=\chi(-1)\,q$.
\end{enumerate}
\end{lemma}
\par
We also require the following deep result due to Katz~\cite[pp.~161--162]{Kat1988}.
\begin{lemma}
\label{lem:sum_of_gauss_sums}
Let $\alpha_1,\dots,\alpha_r,\beta_1,\dots,\beta_s$ be multiplicative characters of $\F_q$ such that $\alpha_1,\dots,\alpha_r$ do not arise by permuting $\beta_1,\dots,\beta_s$. Then
\begin{align*}
\Biggabs{
\sum_{\chi}G(\chi\alpha_1)\cdots\;
G(\chi\alpha_r)\overline{G(\chi\beta_1)}\cdots\;\overline{G(\chi\beta_s)}}\le \max(r,s)\,q^{(r+s+1)/2},
\end{align*}
where the sum runs over all multiplicative characters $\chi$ of $\F_q$.
\end{lemma}
\par
Now let $\psi$ and $\chi$ be multiplicative characters of $\F_q$. The \emph{Jacobi sum} corresponding to $\psi$ and $\chi$ is defined to be
\[
J(\psi,\chi)=\sum_{y\in\F_q}\psi(y)\chi(1-y).
\]
Below we summarise some basic facts about Jacobi sums (see~\cite[Chapter~5]{LidNie1997} or~\cite[Chapter~2]{BerEvaWil1998}, for example).
\begin{lemma}
\label{lem:jacobi}
Let $\psi$ and $\chi$ be multiplicative characters of $\F_q$. Then the following hold.
\begin{enumerate}[(i)]
\setlength{\itemsep}{1ex}
\item $J(\psi,\chi)=0$ if exactly one of $\psi$ or $\chi$ is trivial.
\item $\abs{J(\psi,\chi)}=1$ if $\psi$ and $\chi$ are nontrivial, but $\psi\chi$ is trivial.
\item $\abs{J(\psi,\chi)}=\sqrt{q}$ if all of $\psi$, $\chi$, and $\psi\chi$ are nontrivial.
\item $J(\psi,\chi)\,q=G(\psi)G(\chi)\overline{G(\psi\chi)}$ if $\psi$ and $\chi$ are nontrivial.
\item $J(\psi,\chi)J(\overline{\psi},\chi\psi)=\psi(-1)q$ if $\psi$ and $\chi$ are nontrivial.
\end{enumerate}
\end{lemma}


\section{Gordon-Mills-Welch difference sets}
\label{sec:gmw}

In this section we prove Theorem~\ref{thm:gmw}. If $D$ is a subset of a finite abelian group $G$ and $\chi$ is a character of $G$, we write
\[
\chi(D)=\sum_{d\in D}\chi(d).
\]
The following lemma is a standard (and easily verified).
\begin{lemma}
\label{lem:character_values_diff_set}
A $k$-subset $D$ of a finite abelian group $G$ of order $n$ is a difference set if and only if
\[
\abs{\chi(D)}^2=\frac{k(n-k)}{n-1}
\]
for all nontrivial characters $\chi$ of $G$.
\end{lemma}
\par
The following lemma gives the character values of Gordon-Mills-Welch difference sets and in particular gives an alternative proof of the main result of~\cite{GorMilWel1962} (although, for simplicity, we restrict $q$ to be even).
\begin{lemma}
\label{lem:character_values_gmw}
Let $q>2$ be a power of two and let $\F_s$ be a subfield of $\F_q$. Let~$A$ contain all elements $a\in\F_q$ with $\Tr_{q,s}(a)=1$, let $B$ be a subset of $\F_s^*$, and write $D=\{ab:a\in A,b\in B\}$. Let $\chi$ be a nontrivial character of $\F_q^*$ and let $\chi^*$ be its restriction to $\F_s^*$. Then
\[
\chi(D)=\begin{cases}
\dfrac{\chi^*(B)}{G(\chi^*)}G(\chi) & \text{for $\chi^*$ nontrivial}\\[3ex]
-\dfrac{\chi^*(B)}{s}G(\chi)        & \text{for $\chi^*$ trivial}.
\end{cases}
\]
In particular, if $B$ is a difference set in $\F_s^*$ and $\abs{B}=s/2$, then~$D$ is a difference set with parameters $(q-1,q/2,q/4)$.
\end{lemma}
\begin{proof}
We have
\[
\chi(D)=\sums{a\in\F_q\\\Tr_{q,s}(a)=1}\sum_{b\in B}\chi(ab)=E(\chi)\,\chi^*(B),
\]
where
\[
E(\chi)=\sums{a\in\F_q\\\Tr_{q,s}(a)=1}\chi(a)
\]
is an Eisenstein sum. It is known~\cite[pp.~391/400]{BerEvaWil1998} that
\[
E(\chi)=\begin{cases}
G(\chi)/G(\chi^*) & \text{for $\chi^*$ nontrivial}\\[1ex]
-G(\chi)/s        & \text{for $\chi^*$ trivial},
\end{cases}
\]
which proves the first statement of the lemma. The second statement follows from Lemmas~\ref{lem:character_values_diff_set} and~\ref{lem:gauss}.
\end{proof}
\par
We now prove Theorem~\ref{thm:gmw} by combining Theorem~\ref{thm:mf_from_L_1} with the following result.
\begin{proposition}
\label{pro:L_gmw}
Let $q>2$ be a power of two and let $f$ be a characteristic polynomial of a Gordon-Mills-Welch difference set in $\F_q^*$. Then
\[
\bigabs{L_f(a,b,c)-I_{q-1}(a,b,c)}\le\frac{2q^{5/2}}{(q-1)^3}
\]
for all $a,b,c\in\Z/(q-1)\Z$.
\end{proposition}
\begin{proof}
By definition, there exists a proper subfield $\F_s$ of $\F_q$ such that the underlying Gordon-Mills-Welch difference set is
\[
D=\{ab:a\in A,b\in B\},
\]
where $A$ contains all elements $a\in\F_q$ with $\Tr_{q,s}(a)=1$ and $B$ is a difference set in $\F_s^*$ with $\abs{B}=s/2$. Let~$\theta$ be a generator for $\F_q^*$ such that 
\[
f(z)=\sum_{j=0}^{q-2}\1_D(\theta^j)z^j.
\]
Let $\xi$ be the multiplicative character of $\F_q$ given by $\xi(\theta)=e^{2\pi i/(q-1)}$. It is readily verified that
\[
f(e^{2\pi ik/(q-1)})=\begin{cases}
1         & \text{for $k\equiv 0\pmod n$}\\[1ex]
2\xi^k(D) & \text{for $k\not\equiv 0\pmod n$}.
\end{cases}
\]
It then follows from Lemmas~\ref{lem:character_values_gmw} and~\ref{lem:gauss}~(i) that
\[
f(e^{2\pi ik/(q-1)})=C_kG(\xi^k),
\]
where $C_k$ has unit magnitude for all $k$ and depends only on $k$ modulo $s-1$. Therefore
\begin{equation}
L_f(a,b,c)=\frac{1}{(q-1)^3}\sum_{\chi}G(\chi)G(\chi\xi^a)\overline{G(\chi\xi^b)G(\chi\xi^c)}\;C_\chi(a,b,c),   \label{eqn:L_gauss}
\end{equation}
where the sum is over all multiplicative characters $\chi$ of $\F_q$ and $C_\chi(a,b,c)$ has unit magnitude. Using Lemma~\ref{lem:gauss}, we obtain
\[
L_f(a,b,c)=\begin{cases}
1+\frac{q-2}{(q-1)^2} & \text{for $a=b=c=0$}\\[1ex]
1-\frac{1}{(q-1)^2}   & \text{for $\{0,a\}=\{b,c\}$ and $a\ne 0$},
\end{cases}
\]
which proves the desired result in the case that $I_{q-1}(a,b,c)=1$.
\par
Now assume that $\{0,a\}\ne\{b,c\}$, so that $I_{q-1}(a,b,c)=0$. We need to show that
\begin{equation}
\abs{L_f(a,b,c)}\le \frac{2q^{5/2}}{(q-1)^3}.   \label{eqn:L_GMW_small}
\end{equation}
Let $H$ be the subgroup of index $s-1$ of the character group of $\F_q^*$ and note that $H$ is not the trivial group since, by assumption, $s<q$. Then $C_\chi(a,b,c)$ is constant when $\chi$ ranges over a coset of $H$. Since $C_\chi(a,b,c)$ has unit magnitude, we find from~\eqref{eqn:L_gauss} and the triangle inequality that
\begin{equation}
\abs{L_f(a,b,c)}\le\frac{s-1}{(q-1)^3}\,\max_{\phi}\,\Biggabs{\sum_{\chi\in H}\,G(\chi\phi)G(\chi\phi\xi^a)\overline{G(\chi\phi\xi^b)G(\chi\phi\xi^c)}},   \label{eqn:L_sum_H}
\end{equation}
where the maximum is over all multiplicative characters $\phi$ of $\F_q$. By the definition~\eqref{eqn:def_gauss} of a Gauss sum over $\F_q$, the inner sum can be written as
\[
\sum_{w,x,y,z\in \F_q^*}(-1)^{\Tr_{q,2}(w+x+y+z)}\,\xi^a(x)\overline{\xi^b(y)\xi^c(z)}\,\phi\Big(\frac{wx}{yz}\Big)\sum_{\chi\in H}\chi\Big(\frac{wx}{yz}\Big).
\]
For each $w,x,y,z\in\F_q^*$, we have
\[
\frac{s-1}{q-1}\sum_{\chi\in H}\chi\Big(\frac{wx}{yz}\Big)=\frac{1}{q-1}\sum_\chi\chi\Big(\frac{wx}{yz}\Big),
\]
where the sum on the right-hand side is over all multiplicative characters of~$\F_q$, since both sides equal either $0$ or $1$ depending on whether $wx=yz$ or not. Therefore, we can rewrite the inner sum of~\eqref{eqn:L_sum_H} as
\[
\frac{1}{s-1}\sum_\chi\,G(\chi\phi)G(\chi\phi\xi^a)\overline{G(\chi\phi\xi^b)G(\chi\phi\xi^c)},
\]
where $\chi$ now runs over all multiplicative characters of~$\F_q$. The magnitude of this expression is at most $\frac{2}{s-1}q^{5/2}$ by Lemma~\ref{lem:sum_of_gauss_sums}. Substitute into~\eqref{eqn:L_sum_H} to conclude that~\eqref{eqn:L_GMW_small} holds, as required.
\end{proof}


\section{Sidelnikov sets}
\label{sec:sidelnikov}

In this section we prove Theorem~\ref{thm:sidelnikov} by combining Theorem~\ref{thm:mf_from_L_2} with the following result.
\begin{proposition}
Let $q$ be an odd prime power and let $f$ be a characteristic polynomial of a Sidelnikov set in $\F_q^*$. Then
\[
\bigabs{L_f(a,b,c)-(I_{q-1}(a,b,c)+K_{q-1}(a,b,c))}\le \frac{23q^{5/2}}{(q-1)^3}
\]
for all $a,b,c\in\Z/(q-1)\Z$.
\end{proposition}
\begin{proof}
Let $\eta$ be the quadratic character of $\F_q$. By the definition of a Sidelnikov set~\eqref{eqn:def_sidelnikov}, there exists a generator~$\theta$ of $\F_q^*$ such that
\[
f(z)=z^{\frac{q-1}{2}}+\sum_{j=0}^{q-2}\eta(\theta^j+1)z^j,
\]
where, as usual, $\eta(0)=0$. Let $\xi$ be the multiplicative character of $\F_q$ given by $\xi(\theta)=e^{2\pi i/(q-1)}$. Then we have, for all $k\not\equiv 0\pmod{q-1}$,
\begin{align*}
f(e^{2\pi i k/(q-1)})&=(-1)^k+\sum_{j=0}^{q-2}\eta(\theta^j+1)\xi^k(\theta^j)\\
&=(-1)^k+\sum_{y\in\F_q}\eta(y+1)\xi^k(y)\\
&=(-1)^k+\xi^k(-1)\sum_{y\in\F_q}\eta(y)\xi^k(1-y)\\
&=(-1)^k(1+J(\eta,\xi^k)).
\end{align*}
On the other hand we have $f(1)=1-\eta(1)=0$. Therefore
\begin{equation}
L_f(a,b,c)=\frac{(-1)^{a+b+c}}{(q-1)^3}\sum_{\chi}J(\eta,\chi)J(\eta,\chi\xi^a)\overline{J(\eta,\chi\xi^b)J(\eta,\chi\xi^c)}+\Delta,   \label{eqn:L_sidelnikov}
\end{equation}
where the sum is over all multiplicative characters $\chi$ of $\F_q$ and $\abs{\Delta}\le 15q^{3/2}/(q-1)^2$, using Lemma~\ref{lem:jacobi}. If one of $b$ and $c$ equals $a$ and the other is zero, then by Lemma~\ref{lem:jacobi} the sum in~\eqref{eqn:L_sidelnikov} is between $(q-5)q^2$ and $(q-2)q^2$. If $a=(q-1)/2$ and $b=c+(q-1)/2$, then $\xi^a=\eta$ and $\xi^b=\xi^c\eta$ and by Lemma~\ref{lem:jacobi} (in particular (v)) the sum~\eqref{eqn:L_sidelnikov} is again at least $(q-5)q^2$ and at most $(q-2)q^2$. Since
\[
\frac{(q-5)q^2}{(q-1)^3}=1-\frac{2q^2+3q-1}{(q-1)^3},
\]
this establishes the cases in which either $I_{q-1}(a,b,c)$ or $K_{q-1}(a,b,c)$ equals~$1$.
\par
Now assume that $(a,b,c)$ is such that $I_{q-1}(a,b,c)$ and $K_{q-1}(a,b,c)$ are both zero. Equivalently, the multisets
\begin{equation}
\{\xi^0,\xi^a,\xi^b\eta,\xi^c\eta\}\quad\text{and}\quad\{\eta,\xi^a\eta,\xi^b,\xi^c\}   \label{eqn:sets_chars}
\end{equation}
are distinct. Use Lemmas~\ref{lem:jacobi} and~\ref{lem:gauss} to see that the sum in~\eqref{eqn:L_sidelnikov} equals
\begin{equation}
\frac{1}{q^2}\sums{\chi}G(\chi)G(\chi\xi^a)G(\chi\eta\xi^b)G(\chi\eta\xi^c)\overline{G(\chi\eta)G(\chi\eta\xi^a)G(\chi\xi^b)G(\chi\xi^c)}   \label{eqn:L_sidelnikov_gauss}
\end{equation}
plus an error term of magnitude at most $4q^{3/2}$, where the sum is over all multiplicative characters $\chi$ of $\F_q$. Since the multisets~\eqref{eqn:sets_chars} are distinct, we can apply Lemma~\ref{lem:sum_of_gauss_sums} to see that~\eqref{eqn:L_sidelnikov_gauss} is at most~$4q^{5/2}$. This shows that
\enlargethispage{4ex}
\[
\abs{L_f(a,b,c)}\le \frac{23q^{5/2}}{(q-1)^3},
\]
as required.
\end{proof}


\section{Cyclotomic constructions}
\label{sec:residues}

In this section we prove Theorem~\ref{thm:mf_cyclotomic} and Corollaries~\ref{cor:fourth_order} and~\ref{cor:sixth_order}. We shall use the following notation. Let $m$ be an even positive integer and let~$p$ be a prime satisfying $p\equiv 1\pmod m$. Let $\omega$ be a primitive element of $\F_p$ and let $C_0,C_1,\dots,C_{m-1}$ be the cyclotomic classes of $\F_p$ of order $m$ with respect to $\omega$. Let $S$ be an $m/2$-element subset of $\{0,1,\dots,m-1\}$ and let $D$ be the union of the $m/2$ cyclotomic classes $C_s$ with $s\in S$. Taking $1$ as a generator for the additive group of $\F_p$, a characteristic polynomial of $D$ is
\begin{equation}
f(z)=\sum_{j=0}^{p-1}\1_D(j)z^j.   \label{eqn:f_with_g_1}
\end{equation}
The following lemma gives the evaluations of $f$ at $p$-th roots of unity.
\begin{lemma}
\label{lem:char_val_residue_poly}
Assume the notation as above and let $\chi$ be a multiplicative character of $\F_p$ of order $m$. Then
\[
f(e^{2\pi ik/p})=\frac{2}{m}\sum_{j=1}^{m-1}G(\chi^j)\,\overline{\chi^j(k)}\;\sum_{s\in S}\;\overline{\chi^j(\omega^s)}-1.
\]
\end{lemma}
\begin{proof}
Since $f(1)=\abs{D}-\abs{\F_p\setminus D}=-1$, the result holds for $k\equiv 0\pmod p$, so assume that $k\not\equiv 0\pmod p$. By definition we have
\begin{align*}
f(e^{2\pi ik/p})&=\sum_{y\in D}e^{2\pi iky/p}-\sum_{y\in\F_p\setminus D}e^{2\pi iky/p}\\
&=2\sum_{y\in D}e^{2\pi iky/p}-\sum_{y\in\F_p}e^{2\pi iky/p}\\
&=2\sum_{y\in D}e^{2\pi iky/p}\\
&=2\sum_{s\in S}\sum_{y\in C_s}e^{2\pi iky/p}.
\end{align*}
Writing $h=\frac{p-1}{m}$, the inner sum 
can be written as
\begin{align*}
\sum_{y\in C_s}e^{2\pi iky/p}&=\sum_{j=0}^{h-1}e^{2\pi ik\omega^{mj+s}/p}\\
&=\frac{1}{m}\left(\sum_{y\in\F_p}e^{2\pi ik\,\omega^sy^m/p}-1\right).
\end{align*}
Since $\sum_{j=0}^{m-1}\chi^j(y)$ equals $m$ if $y$ is an $m$-th power and equals zero otherwise, we have, for each $a\in\F_p^*$,
\begin{align*}
\sum_{y\in\F_p}e^{2\pi iay^m/p}
&=\sum_{y\in\F_p}e^{2\pi iay/p}\,\sum_{j=0}^{m-1}\chi^j(y)\\
&=\sum_{j=0}^{m-1}\sum_{y\in\F_p}e^{2\pi iy/p}\,\chi^j(y)\,\overline{\chi^j(a)}.
\end{align*}
For $j=0$, the inner sum equals zero, so we can let the outer sum start with $j=1$. Then all involved  multiplicative characters are nontrivial and we can restrict the summation range of the inner sum to $\F_p^*$. Therefore
\[
\sum_{y\in\F_p}e^{2\pi iay^m/p}=\sum_{j=1}^{m-1}G(\chi^j)\,\overline{\chi^j(a)},
\]
which gives the desired result.
\end{proof}
\par
Our next result estimates $L_f$ for $f$ given in~\eqref{eqn:f_with_g_1} at all points, but $(0,0,0)$.
\begin{proposition}
\label{pro:L_cyclotomic}
With the notation as above, we have
\[
\bigabs{L_f(a,b,c)-(I_p(a,b,c)+\nu J_p(a,b,c))}\le 18(m-1)^4p^{-1/2}
\]
for all $a,b,c\in\Z/p\Z$ with $(a,b,c)\ne(0,0,0)$, where
\[
\nu=\begin{cases}
1                  & \text{for $\frac{p-1}{m}$ even}\\[1ex]
(\frac{4N}{m}-1)^2 & \text{for $\frac{p-1}{m}$ odd}
\end{cases}
\]
and
\[
N=\bigabs{\{(s,s')\in S\times S:s-s'=m/2\}}.
\]
\end{proposition}
\begin{proof}
Let $\chi$ be a multiplicative character of $\F_p$ of order $m$ and write
\[
K(\chi^j)=\frac{2}{m}\sum_{s\in S}\overline{\chi^j(\omega^s)}.
\]
From Lemma~\ref{lem:char_val_residue_poly} we find that 
\begin{equation}
f(e^{2\pi ik/p})=\sum_{j=1}^{m-1}G(\chi^j)K(\chi^j)\overline{\chi^j(k)}-1.   \label{eqn:f_pth_root}
\end{equation}
Hence $L_f(a,b,c)$ equals
\begin{multline}
\frac{1}{p^3}\sum_{j_1,j_2,j_3,j_4=1}^{m-1}G(\chi^{j_1})G(\chi^{j_2})\overline{G(\chi^{j_3})G(\chi^{j_4})}\,K(\chi^{j_1})K(\chi^{j_2})\overline{K(\chi^{j_3})K(\chi^{j_4})}\\
\times\sum_{k\in\F_p}\overline{\chi^{j_1}(k)\chi^{j_2}(k+a)}\chi^{j_3}(k+b)\chi^{j_4}(k+c)+\Delta,   \label{eqn:L_residue_poly}
\end{multline}
where $\abs{\Delta}\le 15(m-1)^4p^{-1/2}$, using that the magnitude of the sum on the right-hand side of~\eqref{eqn:f_pth_root} is at most $(m-1)p^{1/2}$ by Lemma~\ref{lem:gauss}. First consider the case that $b=0$ and $c=a\ne 0$, so that $I_p(a,b,c)=1$ and $J_p(a,b,c)=0$. Then the inner sum in~\eqref{eqn:L_residue_poly} is
\[
\sum_{k\in\F_p}\chi^{j_3-j_1}(k)\chi^{j_4-j_2}(k+a).
\]
This sum either has magnitude at most $\sqrt{p}$ by the Weil bound (see~\cite[Theorem~5.41]{LidNie1997}, for example) or equals $p$. Since $a\ne 0$, the latter case occurs if and only if $j_1\equiv j_3\pmod m$ and $j_2\equiv j_4\pmod m$. Therefore $L_f(a,0,a)$ equals 
\[
\frac{1}{p^2}\left(\sum_{j=1}^{m-1}\abs{G(\chi^j)}^2\,\abs{K(\chi^j)}^2\right)^2
\]
plus an error term of magnitude at most $16(m-1)^4p^{-1/2}$. Then we find from Lemma~\ref{lem:gauss} and 
\[
\sum_{j=1}^{m-1}\bigabs{K(\chi^j)}^2=1
\]
that the desired result holds for $b=0$ and $c=a\ne 0$. The case $c=0$ and $b=a\ne 0$ is completely analogous.
\par
Now assume that $a=0$ and $c=b\ne 0$, so that $I_p(a,b,c)=0$ and $J_p(a,b,c)=1$. Then the inner sum in~\eqref{eqn:L_residue_poly} equals
\[
\sum_{k\in\F_p}\overline{\chi^{j_1+j_2}(k)}\chi^{j_3+j_4}(k+b).
\]
As before, this sum either has magnitude at most $\sqrt{p}$ or equals~$p$, where the latter case occurs if and only if $j_1\equiv -j_2\pmod m$ and $j_3\equiv -j_4\pmod m$. Hence $L_f(0,b,b)$ equals
\begin{equation}
\frac{1}{p^2}\Biggabs{\sum_{j=1}^{m-1}G(\chi^j)G(\overline{\chi^j})\,K(\chi^j)K(\overline{\chi^j})}^2   \label{eqn:L_residue_poly_aa}
\end{equation}
plus an error term of magnitude at most $16(m-1)^4p^{-1/2}$. From Lemma~\ref{lem:gauss} we find that~\eqref{eqn:L_residue_poly_aa} equals
\[
\Biggabs{\sum_{j=1}^{m-1}\chi^j(-1)\,\abs{K(\chi^j)}^2}^2=\Bigg(\sum_{j=1}^{m-1}(-1)^{\frac{j(p-1)}{m}}\Biggabs{\frac{2}{m}\,\sum_{s\in S} e^{2\pi ijs/m}}^2\Bigg)^2.
\]
A standard calculation then shows that this expression equals $\nu$. This proves the desired result in the case that $a=0$ and $c=b\ne 0$.
\par
Now assume that $0,a,b,c$ do not form two pairs of equal elements. In this case, we invoke the Weil bound again to conclude that the inner sum in~\eqref{eqn:L_residue_poly} is at most $3\sqrt{p}$ in magnitude. Therefore we have by Lemma~\ref{lem:gauss}
\[
\abs{L_f(a,b,c)}\le 18(m-1)^4p^{-1/2},
\]
which completes the proof.
\end{proof}
\par
\begin{proof}[Proof of Theorem~\ref{thm:mf_cyclotomic}]
Without loss of generality, we may choose $1$ as a generator for the additive group of $\F_p$ and take~\eqref{eqn:f_with_g_1} as a characteristic polynomial of $D$ (if the generator is $v$, then replace $D$ by $v^{-1}D$).
\par
We shall deduce Theorem~\ref{thm:mf_cyclotomic} from Theorem~\ref{thm:mf_from_L_1}. Proposition~\ref{pro:L_cyclotomic} takes care of all values of $L_f(a,b,c)$ in the condition of Theorem~\ref{thm:mf_from_L_1}, except when $(a,b,c)=(0,0,0)$. We shall show that our assumption~\eqref{eqn:cond_Ru} takes care of the latter case. Writing
\[
R_u=\sum_{y\in\F_p}\1_D(y)\1_D(y+u),
\]
a standard calculation gives
\[
\abs{f(e^{2\pi ik/p})}^2=\sum_{u\in\F_p}R_u\,e^{-2\pi iku/p}
\]
and therefore, by Parseval's identity,
\[
\frac{1}{p}\sum_{k\in\F_p}\abs{f(e^{2\pi ik/p})}^4=\sum_{u\in\F_p}R_u^2.
\]
A counting argument shows that
\[
R_u=4\,\abs{(D+u)\cap D}-(p-2).
\]
Therefore, since $2\abs{D}=p-1$, we find that  $L_f(0,0,0)$ equals
\[
\frac{1}{p^3}\sum_{k\in\F_p}\abs{f(e^{2\pi ik/p})}^4=
1+\frac{1}{p^2}\sum_{u\in\F_p^*}\big(4\,\abs{(D+u)\cap D}-(p-2)\big)^2.
\]
Now our assumption~\eqref{eqn:cond_Ru} together with Proposition~\ref{pro:L_cyclotomic} imply that the condition of Theorem~\ref{thm:mf_from_L_1} is satisfied, which proves Theorem~\ref{thm:mf_cyclotomic}.
\end{proof}
\par
Next we show how to deduce Corollaries~\ref{cor:fourth_order} and~\ref{cor:sixth_order} from Theorem~\ref{thm:mf_cyclotomic}. First consider Corollary~\ref{cor:fourth_order}. As explained in Section~\ref{sec:results}, we just need to consider the case that $D=C_0\cup C_1$. The cyclotomic numbers of order four have been already determined by Gauss and can be found, for example, in~\cite[Theorem 2.4.1]{BerEvaWil1998}. These numbers depend on the representation $p=x^2+4y^2$ and on the parity of $(p-1)/4$. They also depend on the choice of the primitive element in $\F_p$ used to define the cyclotomic classes, but this is encapsulated in the fact that $y$ is only unique up to sign. Using the cyclotomic numbers of order four, we obtain the numbers $\abs{(D+u)\cap D}$, shown in Table~\ref{tab:cyc_num_4}. For example, if $u\in C_1$, then $u^{-1}\in C_3$ and $\abs{(D+u)\cap D}$ equals
\[
\abs{(C_3+1)\cap C_3}+\abs{(C_3+1)\cap C_0}+\abs{(C_0+1)\cap C_3}+\abs{(C_0+1)\cap C_0}.
\]
From the data in Table~\ref{tab:cyc_num_4} we conclude that the assumption $y^2(\log p)^3/p\to 0$ as $p\to\infty$ in Corollary~\ref{cor:fourth_order} implies the condition~\eqref{eqn:cond_Ru} in Theorem~\ref{thm:mf_cyclotomic}. It is also readily verified that $\nu=1$, which proves Corollary~\ref{cor:fourth_order}.
\begin{table}
\caption{The numbers $4\abs{(D+u)\cap D}-(p-2)$ for $D=C_0\cup C_1$ and primes~$p$ of the form $p=x^2+4y^2$.}
\label{tab:cyc_num_4}
\renewcommand{\arraystretch}{1.2}
\begin{tabular}{c||c|c}
\hline
$D$         & $\frac{p-1}{4}$ even & $\frac{p-1}{4}$ odd\\ \hline\hline
$u\in C_0$  & $-3+2y$ & $-1-2y$\\
$u\in C_1$  & $-3-2y$ & $-1+2y$\\
$u \in C_2$ & $ 1+2y$ & $-1-2y$\\
$u \in C_3$ & $ 1-2y$ & $-1+2y$\\\hline
\end{tabular}
\end{table}
\par
Now consider Corollary~\ref{cor:sixth_order}. Then it suffices to consider the cases that $D$ is one of the following sets
\[
C_0\cup C_1\cup C_2,\quad C_0\cup C_1\cup C_3,\quad C_0\cup C_1\cup C_4.
\]
The cyclotomic numbers of order six have been determined by Dickson~\cite{Dic1935} (see also~\cite{Hal1956} for $(p-1)/6$ odd and~\cite{Whi1960} for $(p-1)/6$ even). These numbers depend on the representation of $p$ as a sum of a square and three times a square (every prime congruent to $1$ modulo~$3$ can be represented in this way) and on the cubic character of~$2$. Since~$p$ is of the form $x^2+27y^2$, we know~\cite[Theorem~2.6.4]{BerEvaWil1998} that $2$ is a cube in $\F_p$. In this case, the numbers $\abs{(D+u)\cap D}$ are given in Tables~\ref{tab:cyc_num_6_odd} and~\ref{tab:cyc_num_6_even} (again $y$ is only unique up to sign, corresponding to different primitive elements in~$\F_p$). We again conclude that the assumption $y^2(\log p)^3/p\to 0$ as $p\to\infty$ in Corollary~\ref{cor:sixth_order} implies the condition~\eqref{eqn:cond_Ru} in Theorem~\ref{thm:mf_cyclotomic}. The proof of Corollary~\ref{cor:sixth_order} is completed by checking that $\nu=1$ if $D$ is of Paley type and $\nu=1/9$ if $D$ is of 
Hall type.
\begin{table}
\caption{The numbers $4\abs{(D+u)\cap D}-(p-2)$ for primes of the form $x^2+27y^2$ and $(p-1)/6$ odd.}
\label{tab:cyc_num_6_odd}
\renewcommand{\arraystretch}{1.2}
\begin{tabular}{c||c|c|c}
\hline
$D$        & $C_0\cup C_1\cup C_2$ & $C_0\cup C_1\cup C_3$ & $C_0\cup C_2\cup C_3$ \\\hline\hline
$u\in C_0$ & $-1+8y$ & $-3+2y$ & $-3-2y$\\
$u\in C_1$ & $-1$    & $-1$    & $1+2y$ \\
$u\in C_2$ & $-1-8y$ & $1-2y$  & $-1$   \\
$u\in C_3$ & $-1+8y$ & $-3+2y$ & $-3-2y$\\
$u\in C_4$ & $-1$    & $-1$    & $1+2y$ \\
$u\in C_5$ & $-1-8y$ & $1-2y$  & $-1$   \\\hline
\end{tabular}
\end{table}
\begin{table}
\caption{The numbers $4\abs{(D+u)\cap D}-(p-2)$ for primes of the form $x^2+27y^2$ and $(p-1)/6$ even.}
\label{tab:cyc_num_6_even}
\renewcommand{\arraystretch}{1.2}
\begin{tabular}{c||c|c|c}
\hline
$D$        & $C_0\cup C_1\cup C_2$ & $C_0\cup C_1\cup C_3$ & $C_0\cup C_2\cup C_3$ \\\hline\hline
$u\in C_0$ & $-3+8y$ & $-3+6y$ & $-3+2y$\\
$u\in C_1$ & $-3$    & $-3-4y$ & $1-2y$\\
$u\in C_2$ & $-3-8y$ & $1+2y$  & $-3+4y$\\
$u\in C_3$ & $1+8y$  & $-3-2y$ & $-3-6y$\\
$u\in C_4$ & $1$     & $1+4y$  & $1+6y$\\
$u\in C_5$ & $1-8y$  & $1-6y$  & $1-4y$\\\hline
\end{tabular}
\end{table}




\providecommand{\bysame}{\leavevmode\hbox to3em{\hrulefill}\thinspace}
\providecommand{\MR}{\relax\ifhmode\unskip\space\fi MR }
\providecommand{\MRhref}[2]{%
  \href{http://www.ams.org/mathscinet-getitem?mr=#1}{#2}
}
\providecommand{\href}[2]{#2}

\end{document}